\documentclass[10pt,final,oneside]{amsart}

\usepackage[margin=1in,includehead]{geometry}
\usepackage[secnum]{harrstyle}
\usepackage{verbatim}
\usepackage{setspace}
\usepackage{amssymb}
\usepackage[vcentermath]{youngtab}
\usepackage{epigraph}
\usepackage{diagrams}
\usepackage{enumitem}
\usepackage{wrapfig}
\usepackage{diagrams}
\usepackage{mathtools}

\newcommand{\cA}{\cal A}
\newcommand{\Sp}{\mathrm{Sp}}

\renewcommand{\S}{S}
\newcommand{\PGL}{\mathrm{PGL}}

\newcommand{\Stab}{\mathrm{Stab}}
\newcommand{\M}{\cal M}
\newcommand{\Fix}{{\rm Fix}}

\newcommand{\llangle}{\langle\langle}
\newcommand{\rrangle}{\rangle\rangle}
\newcommand{\grad}{\mathrm{grad}}
\newcommand{\SL}{\mathrm{SL}}
\newcommand{\T}[1]{{\rm Teich}(\Sigma_{#1})}

\newcommand{\Mod}[1]{{\rm Mod}(\Sigma_{#1})}
\newcommand{\HH}[1]{{\rm SMod}(\Sigma_{#1})}

\newcommand{\Orb}{{\rm Orb}}
\newcommand{\sympkernel}{\tilde K}
\renewcommand{\im}{\rm Im}
\newcommand{\Heis}[1]{{\rm Heis_{#1}}}
\newcommand{\Hess}{{\rm Hess\;}}
\newcommand{\bc}{{\bf c}}
\newcommand{\Diff}{{\rm Diff}}

\renewcommand{\P}{\mathbb P}

\newarrow{ToX}{}{.}{}{.}{>}

\numberwithin{equation}{section}

\newtheoremstyle{named}{}{}{\itshape}{}{\bfseries}{.}{.5em}{\thmnote{#3}}
\theoremstyle{named}
\newtheorem*{namedtheorem}{Theorem}

\begin{document}

\onehalfspacing

\title{The kernel of the monodromy of the universal family of degree $d$ smooth plane curves}

\author{Reid Monroe Harris}
\date{}

\renewcommand{\textflush}{center}
\renewcommand{\epigraphflush}{center}
		\renewcommand{\epigraphrule}{0pt}

\maketitle	

\begin{abstract}
We consider the parameter space $\cal U_d$ of smooth plane curves of degree $d$. The universal smooth plane curve of degree $d$ is a fiber bundle $\cal E_d\to\cal U_d$ with fiber diffeomorphic to a surface $\Sigma_g$. This bundle gives rise to a monodromy homomorphism $\rho_d:\pi_1(\cal U_d)\to\Mod{g}$, where $\Mod{g}:=\pi_0(\Diff^+(\Sigma_g))$ is the mapping class group of $\Sigma_g$. The main result of this paper is that the kernel of $\rho_4:\pi_1(\cal U_4)\to\Mod{3}$ is isomorphic to $F_\infty\times\Z/3\Z$, where $F_\infty$ is a free group of countably infinite rank. In the process of proving this theorem, we show that the complement $\T{g}\setminus\cal H_g$ of the hyperelliptic locus $\cal H_g$ in Teichm\"uller space $\T{g}$ has the homotopy type of an infinite wedge of spheres. As a corollary, we obtain that the moduli space of plane quartic curves is aspherical. The proofs use results from the Weil-Petersson geometry of Teichm\"uller space together with results from algebraic geometry.
\end{abstract}

\section{Introduction}

Let $\P\left(\Sym^d\left(\C^3\right)\right)=\P^N$, where $N=\binom{d+2}{2}-1$, be the parameter space of plane curves of degree $d>0$. Elements of $\P^N$ are homogeneous degree $d$ polynomials in variables $x,y,z$. Let $\cal U_d$ denote the {\it parameter space of smooth} plane curves of degree $d$. More precisely, $\cal U_d=\P^N\setminus\Delta_d$ is the complement of the {\it discriminant locus} $\Delta_d\subset\P^N$ which is the set of polynomials $f$ such that the curve $V(f)=\{p\in\P^2:f(p)=0\}$ is singular.

The {\it universal smooth plane curve of degree $d$} is the fiber bundle $\cal E_d\to\cal U_d$ defined by
\begin{align*}
\cal E_d:=\{(f,p)\in\cal U_d\times\P^2: f(p)=0\}&\to\cal U_d\\
(f,p)&\mapsto f
\end{align*}
There exists a monodromy homomorphism $$\rho_d:\pi_1\left(\cal U_d\right)\to\Mod{g},$$ where $\Mod{g}:=\pi_0({\rm Diff}^+(\Sigma_g))$ is the mapping class group. We omit reference to the basepoint in $\pi_1\left(\cal U_d\right)$, however, it can be taken to be the Fermat curve $f_F(x,y,z)=x^d+y^d+z^d=0$. The homomorphism $\rho_d$ is called the {\it geometric monodromy of the universal smooth plane curve of degree $d$}. A finite presentation for $\pi_1(\cal U_d)$ has been given by L\"onne \cite[Main Theorem]{Lon}.

Two natural questions are to determine the image $\im(\rho_d)$ and kernel $K_d:=\ker(\rho_d)$. Dolgachev and Libgober have given a description of $\pi_1(\cal U_3)$ as an extension $$0\to \Heis{3}(\Z/3\Z)\to\pi_1(\cal U_3)\xrightarrow{\rho_3}\Mod{1}\to 0$$ \cite[Exact Squence 4.8]{DL} of $\Mod{1}$ by the $\Z/3\Z$-points of the 3-dimensional Heisenberg group \cite[Page 12]{DL} $$\Heis{3}(\Z/3\Z):=\left\{\left(\begin{array}{ccc} 1&*&*\\0&1&*\\0&0&1\end{array}\right): *\in\Z/3\Z\right\}$$ The action $\Mod{1}\circlearrowleft H_1\left(\Heis{3}(\Z/3\Z);\Z\right)\cong(\Z/3\Z)^2$ is the action on the Weierstra\ss\, points of the elliptic curve. This action is exactly the composition $\Mod{1}\xrightarrow{\Psi_1}\SL_2(\Z)\to\SL_2(\Z/3\Z)$, where $\Psi_1:\Mod{1}\cong\SL_2(\Z)$ is the action on $H_1(\Sigma_1;\Z)$, see \cite[Theorem 2.5]{FM}, and $\SL_2(\Z)\to\SL_2(\Z/3\Z)$ is the natural projection.

For higher degrees $d\ge 4$, there is an exact sequence $$0\to K_d\to \pi_1(\cal U_d)\xrightarrow{\rho_d} \Mod{g}.$$ The map $\rho_d$ is, in general, not surjective. However, Salter \cite[Theorem A]{Salt2} has shown that $\im(\rho_d)$ always has finite index in $\Mod{g}$. For $d=4$, Kuno has shown that $\im(\rho_4)=\Mod{3}$ and that $K_4$ is infinite \cite[Proposition 6.3]{K}. For $d=5$, Salter \cite[Theorem A]{Salt1} shows that $\im(\rho_5)$ is the stabilizer $\Mod{6}[\phi]$ of a certain spin structure $\phi$ on $\Sigma_6$, the spin structure $\phi=e^*\cal O(1)$ induced on $\Sigma_6$ by its embedding $e:\Sigma_6\to\P^2$ as a plane curve. For odd $d\ge 5$, Salter shows that the monodromy group $\im(\rho_d)$ is the stabilizer of a spin structure on $\Sigma_g$, for $g=\binom{d+1}{2}$. For even $d\ge 6$, $\im(\rho_d)$ is only known to be finite index in this stabilizer, hence in $\Mod{g}$ \cite[Theorem A]{Salt2}.

Another result in this vein $\pi_1(\cal U_d)$ can be found in \cite{CT}. Recall that $\Mod{g}$ acts on $H_1(\Sigma_g;\Z)$ preserving the intersection form. This gives rise to the {\it symplectic representation} $\Psi_g:\Mod{g}\to\Sp_{2g}(\Z)$. Consider the composition $$\Psi_g\circ\rho_d:\pi_1(\cal U_d)\to\Sp_{2g}(\Z).$$ This representation is called the {\it algebraic monodromy of the universal smooth plane curve of degree d}. Carlson and Toledo show that $\sympkernel_d:=\ker(\Psi_g\circ\rho_d)$ is {\it large} \cite[Theorem 1.2]{CT}, i.e. there is a homomorphism $\sympkernel_d\to G$ to a noncompact semisimple real algebraic Lie group $G$ with Zariski-dense image.

In this paper we prove the following theorem, which is a refinement of Kuno's theorem \cite[Proposition 6.3]{K} that $K_4$ is infinite. In the statement, $\HH{g}<\Mod{g}$ denotes the centralizer of a fixed hyperelliptic involution, the homotopy class of an order 2 homeomorphism $\tau:\Sigma_g\to\Sigma_g$ which acts on $H_1(\Sigma_g;\Z)$ by multiplication by $-1$.

\begin{thm}\label{M1}
The group $K_4$ is isomorphic to $F_\infty\times\Z/3\Z$, where $F_\infty$ is an infinite rank free group. Moreover, $F_\infty$ has a free generating set in bijection with the set of cosets of the hyperelliptic mapping class group $\HH{3}$, and $$H_1(K_4;\Q)\cong \Q[\Mod{3}/\HH{3}]$$ as $\Mod{3}$-modules.
\end{thm}

The idea for the proof of Theorem \ref{M1} is to exhibit the cover $\cal U_4^{mark}\to\cal U_4$ corresponding to $K_4$ as a principal fiber bundle over the complement $\T{3}\setminus\cal H_3$ of the hyperelliptic locus $\cal H_3$ in Teichm\"uller space $\T{3}$. The following theorem determines the homotopy type of $\T{3}\setminus\cal H_3$.

\begin{thm}\label{M2}
Let $g\ge 3$. The hyperelliptic complement $\T{g}\setminus\cal H_g$ has the homotopy type of a wedge $\displaystyle\bigvee_{i=1}^\infty S^n$ of infinitely many $n$-spheres, where $n=2g-5$.
\end{thm}

From Theorem \ref{M2}, we can conclude that $\cal U_4^{mark}\to\T{3}\setminus\cal H_3$ is trivial and Theorem \ref{M1} follows.

We will also show that the structure of the group $K_d$ is closely related to that of the hyperelliptic mapping class group. The failure of our proof method in Theorem \ref{M1} for degrees $d>4$ is due to the lack of knowledge of the topology of the locus of planar curves in the moduli space of Riemann surfaces; there are many more obstructions to being planar than being hyperelliptic. 

The paper is organized as follows. Section \ref{WPmetric} recalls basic facts about the Weil-Petersson metric on Teichm\"uller space and the hyperelliptic locus. Section \ref{HyperComplement} introduces the geodesic length functions. These will then be used to prove Theorem \ref{M2}. The proof of Theorem \ref{M1} is carried out in section \ref{section:PlaneCurves}.

\subsection{Acknowledgements} The author would like to thank his advisor Benson Farb, and Nick Salter, for their extensive comments on earlier drafts of this paper. The author would additionally like to thank Howard Masur for information about the Weil-Petersson Metric and geometric length functions, Shmuel Weinberger, Madhav Nori, and Scott Wolpert. The author would also like to thank the Jump Trading Mathlab Research Fund for their support.

\section{The Hyperelliptic Locus and the Weil-Petersson Metric}\label{WPmetric}

For the rest of the paper, let $g\ge 2$ unless otherwise stated. In this section we give the necessary background on Teichm\"uller space and its geometry. We review the Weil-Petersson metric on Teichm\"uller space and describe the geometric properties of the hyperelliptic locus in terms of this metric, see Proposition \ref{locus}.

\subsection{Teichm\"uller Space} We recall the basic theory of Teichm\"uller space and of the moduli space of Riemann surfaces of genus $g$. For additional background, see e.g. \cite{FM}. Let $\T{g}$ denote the Teichm\"uller space of genus $g\ge 2$ curves. That is, $\T{g}$ is the set of equivalence classes $[X,h]$ of pairs $(X,h)$, where $X$ is a complex curve of genus $g$ and $h$ is a {\it marking}, i.e. a homeomorphism $\Sigma_g\to X$. Two pairs $(X,h)$ and $(Y,g)$ are equivalent if $h\circ g^{-1}:Y\to X$ is isotopic to a biholomorphism. We will also denote such an equivalence class $[X,h]$ by $\cal X$. The (complex) dimension of $\T{g}$ is $3g-3$.

The mapping class group $\Mod{g}$ acts on $\T{g}$ by $$[f]\cdot [X,h]=[X,h\circ f^{-1}]$$ where $[f]\in\Mod{g}$. This action is properly discontinuous \cite[Theorem 12.2]{FM} so that the quotient space $\cal M_g:=\Mod{g}\backslash\T{g}$, the {\it moduli space of genus $g$ Riemann surfaces}, is an orbifold. Let $\pi:\T{g}\to\cal M_g$ denote the quotient map. The space $\cal M_g$ can also be defined as the space of all complex curves of genus $g$, up to biholomorphism. Note that the orbifold fundamental group $\pi_1^{orb}(\cal M_g)$ of $\cal M_g$ is $\Mod{g}$.

\subsection{Weil-Petersson Metric} In this subsection we recall the Weil-Petersson (WP) metric and some of its properties. The WP metric is a certain K\"ahler metric on $\T{g}$ which gives rise to a Riemannian structure on $\T{g}$. For more on the Weil-Petersson metric, see the survey \cite{Wol3}.

The cotangent space $T^*_{\cal X}\T{g}$ at a point $\cal X=[X,h]\in\T{g}$ can be identified with the space $Q(X)$ of holomorphic quadratic differentials on $X$. Define a (co)metric on $T^*_{\cal X}\T{g}$ by $$\llangle\varphi,\psi\rrangle:=\int_X\varphi\bar\psi (ds^2)^{-1},$$ where $ds^2$ is the hyperbolic metric on $X$ and $(ds^2)^{-1}$ is its dual. The {\it Weil-Petersson (WP) metric} is defined to be the dual of $\llangle\cdot,\cdot\rrangle$.

The WP metric is a $\Mod{g}$-invariant, incomplete \cite[Section 2]{Wol2}, smooth Riemannian metric of negative sectional curvature \cite[Theorem 2]{Tromba}. Teichm\"uller space $\T{g}$ equipped with the WP metric is {\it geodesically convex} \cite[Subsection 5.4]{Wol}, meaning that any two points $\cal X,\cal Y\in\T{g}$ are connected by a unique geodesic. When referring to any metric properties of Teichm\"uller space, we will assume they are with respect to the WP metric unless otherwise stated.

\subsection{Hyperelliptic Locus} A {\it hyperelliptic curve} $X$ is a complex curve equipped with a biholomorphic involution $\tau:X\to X$ such that $X/\tau$ is isomorphic to $\P^1$. Such a map $\tau$, if it exists, is called a {\it hyperelliptic involution}. An element $[\tau]\in\Mod{g}$ is called a {\it hyperelliptic mapping class} if $[\tau]^2=1$ and $\Sigma_g/\tau$ is homeomorphic to $\bb P^1$, or equivalently, if $[\tau]$ acts on $H_1(\Sigma_g;\Z)$ by multiplication by $-1$.

Let $\bar{\cal H_g}\subset\cal M_g$ denote the locus of hyperelliptic curves and let $\cal H_g:=\pi^{-1}(\bar{\cal H_g})$, where $\pi:\T{g}\to\cal M_g$ is the quotient map. The set $\cal H_g$ is called the {\it hyperelliptic locus}. It has (complex) dimension $2g-1$. Note that when $g=3$, the hyperelliptic locus $\cal H_3$ has complex codimension 1 in $\T{g}$.

The following proposition collects some facts that will be useful in later sections.

\begin{prop}\label{locus}
The locus $\cal H_g$ is a complex-analytic submanifold of $\T{g}$. Moreover, $\cal H_g$ has infinitely many connected components (see Figure 1). If $H$ is any component of $\cal H_g$ then $H$ is totally geodesic in $\T{g}$ and $H$ is biholomorphic to $\T{0,2g+2}$, the Teichm\"uller space of a sphere with $2g+2$ punctures. In particular, each component of $\cal H_g$ is contractible.
\end{prop}

\begin{proof}
Let $[\tau]\in\Mod{g}$ be a hyperelliptic mapping class. Then $[\tau]$ acts on $\T{g}$ with fixed set $$\Fix([\tau]):=\{[Y,g]\in\T{g}: [Y,g]=[Y,g\circ\tau]\}.$$

First, we show that $$\cal H_g=\bigcup_{[\tau]\text{ hyperelliptic}}\Fix([\tau]),$$ where the union is taken over all hyperelliptic mapping classes $[\tau]\in\Mod{g}$. If $[X,h]\in\Fix([\tau])$ then $\tau:X\to X$ is isotopic to a biholomorphism $\tau_b$. The map $\tau_b$ must be a hyperelliptic involution, and so $[X,h]\in\cal H_g$. Conversely, if $[X,h]\in\cal H_g$ then there is a hyperelliptic involution $\tau:X\to X$ which is a biholomorphism and so $[X,h]\in\Fix([\tau])$.

If $[\tau]$ and $[\eta]$ are two distinct hyperelliptic mapping classes, then $\Fix([\tau])\cap\Fix([\eta])=\varnothing$. More explicitly, if $[X,h]\in\Fix([\tau])\cap\Fix([\eta])$ then, $[\tau]$ and $[\eta]$ contain biholomorphic representatives $\tau_b,\eta_b:X\to X$. By \cite[Section III.7.9, Corollary 2]{FK}, we must have $\tau_b=\eta_b$.

Each set $\Fix([\tau])$ is totally geodesic in $\T{g}$. This follows from the uniqueness of geodesics in the WP metric: if $\gamma$ is any geodesic with endpoints lying in $\Fix([\tau])$, then $[\tau]\cdot\gamma$ must be another geodesic with the same endpoints as $\gamma$, hence $\gamma$ must be fixed by $\tau$.

For a proof that $\cal H_g$ is a complex-analytic submanifold of $\T{g}$ and that each component is biholomorphic to $\T{0,2g+2}$, we refer the reader to \cite[Section 4.1.5]{Nag}.
\end{proof}

\begin{figure}[h]
\includegraphics[scale=0.6]{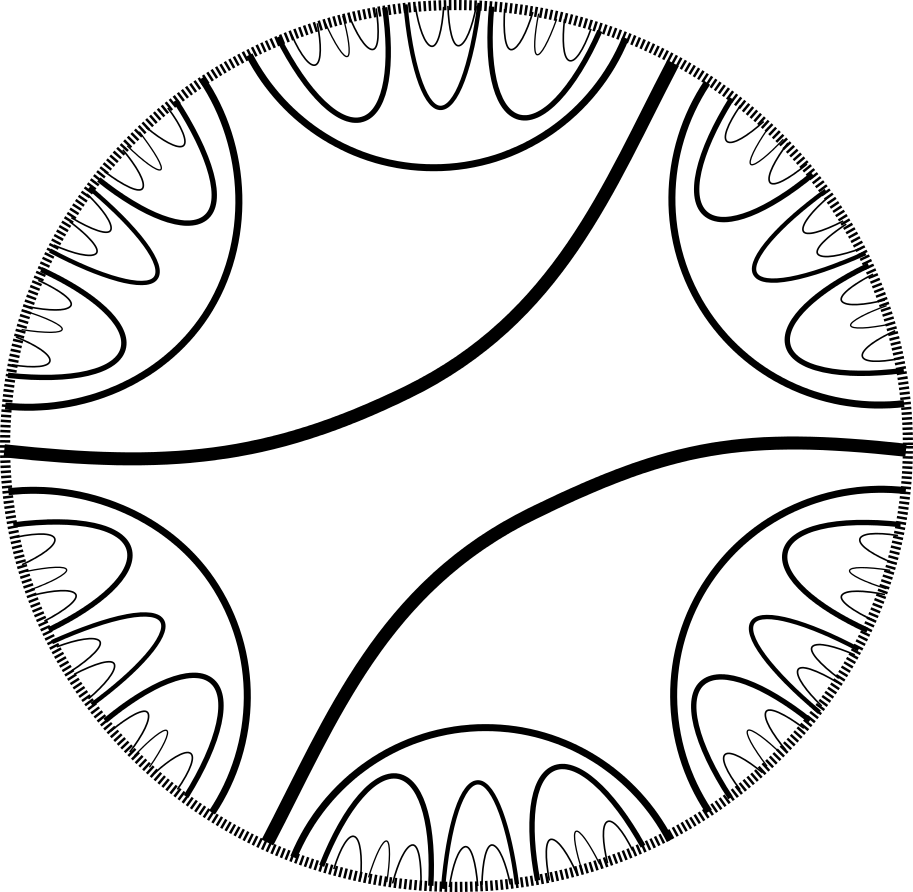}
	\caption{A schematic of the hyperelliptic locus $\cal H_g$ in $\T{g}$. The submanifold $\cal H_g\subset\T{g}$ has infinitely many connected components, each of which is totally geodesic with respect to the Weil-Petersson metric.}
\end{figure}

\section{Homotopy Type of the Hyperelliptic Complement}\label{HyperComplement}

In Section 3.1, we prove, Lemma \ref{Lem1}, the existence of certain Morse functions on $\T{g}$. These functions will be used to prove Theorem \ref{M2} in Section 3.2.

\subsection{Geodesic Length Functions}

This section is devoted to proving the existence of sufficiently well-behaved functions on $\T{g}$.

\begin{lem}\label{Lem1}
Let $g\ge 3$. There exists a function $f:\T{g}\to\R_+$ which satisfies the following properties.
\begin{enumerate}
\item The function $f$ is proper, strictly convex and has positive-definite Hessian everywhere.
\item The function $f$ has a unique critical point in $\T{g}$, denoted $x_0$.
\item For any component $H$ of $\cal H_g$, the restriction $f|_H$ has a unique critical point, denoted $x_H$.
\item Any two critical values are distinct. That is, for any component $H$ of $\cal H_g$, $f(x_H)\not=f(x_0)$. Also, if $H'$ is any other component of $\cal H_g$, then $f(x_H)=f(x_{H'})$ if and only if $H=H'$.
\item The set of critical values $$\{f(x_H):H\text{ is a component of }\cal H_g\}\cup\{f(x_0)\}$$ is a discrete subset of $\R_+$.
\end{enumerate}
In particular, such a function $f$ is Morse on $\T{g}$ and for each component $H$ of $\cal H_g$, the restriction $f|_H$ is Morse.
\end{lem}

\begin{proof}
The function $f$ is built using {\it geodesic length functions}. These functions are defined as follows. Let $\alpha$ be a free homotopy class of simple closed curves on $\Sigma_g$ and let $[X,h]$ be a point in $\T{g}$. Then $h(\alpha)$ is a free homotopy class of simple closed curves in $X$. Recall that $h(\alpha)$ contains a unique geodesic $\gamma$. The {\it geodesic length function} $\ell_\alpha:\T{g}\to\R_+$ associated to $\alpha$ is defined by $$\ell_\alpha(\cal X):=\text{length of the unique geodesic in the free homotopy class $h(\alpha)$ on $X$},$$ where $\cal X=[X,h]$. Any other choice $(X',h')$ of representative of $[X,h]$ would differ from $(X,h)$ by an isometry, hence $\ell_\alpha$ is well-defined. Fix a finite collection $\cal A$ of (homotopy classes of) simple closed curves which fills $\Sigma_g$, and let ${\bc}=(c_\alpha)\in\R_+^\cal A$ be a collection of positive real numbers for each $\alpha\in\cal A$. For each choice of ${\bc}\in\R_+^\cal A$, there is a function $$\cal L_{\cal A,{\bc}}:=\sum_{\alpha\in\cal A}c_\alpha\ell_\alpha:\T{g}\to\R_+.$$ The function $f$ in the statement of the theorem will be defined to be $\cal L_{\cal A,\bc}$ for a specific value of $\bc$.

Wolpert \cite[Theorem 4.6]{Wol} states that for any free homotopy class of simple closed curves $\alpha$ on $\Sigma_g$, the geodesic length function $\ell_\alpha$ has positive-definite Hessian everywhere. In particular, $\ell_\alpha$ is strictly convex along WP geodesics.

Recall that the Hessian operator $\Hess$ is given in local coordinates by $$f\mapsto \left(\frac{\partial^2f}{\partial x^i\partial x^j}+\Gamma^k_{ij}\frac{\partial f}{\partial x^k}\right)dx^i\otimes dx^j,$$ where $\Gamma_{ij}^k$ are the Christoffel symbols given by $g$. Thus, ${\rm Hess}$ is $\R$-linear. It follows that $$\Hess\cal L_{\cal A,{\bc}}=\sum_{\alpha\in\cal A}c_\alpha\cdot\left(\Hess\ell_\alpha\right).$$ For any $v\in T_\cal X\T{g}$, $$\Hess\cal L_{\cal A,{\bc}}(v,v)=\sum_{\alpha\in\cal A}c_\alpha\cdot\left(\Hess\ell_\alpha\right)(v,v)>0$$ and so $\Hess\cal L_{\cal A,{\bc}}$ is positive-definite. This also shows that $\cal L_{\cal A,{\bc}}$ is strictly convex.

Let ${\bf 1}$ denote the element of $\R_+^\cal A$ such that $c_\alpha=1$ for all $\alpha\in\cal A$. For $\bc=(c_\alpha)\in\R_+^\cA$, let $c_{min}:=\min_{\alpha\in\cal A}c_\alpha$. Then, $$c_{min}\cal L_{\cal A,{\bf 1}}\le\cal L_{\cal A,\bc}.$$ Kerckhoff \cite[Lemma 3.1]{Ker} states that the functions $\cal L_{\cal A,{\bf 1}}$ are proper. If $K=[a,b]\subset\R_+$ is compact, then $$(\cal L_{\cal A,\bc})^{-1}(K)\subset (\cal L_{\cal A,{\bf 1}})^{-1}\left[0,b/c_{min}\right],$$ so $(\cal L_{\cal A,\bc})^{-1}(K)$ is a closed subset of a compact set, hence is compact. Thus, $\cal L_{\cal A,\bc}$ is proper. This proves (1) in the statement of the theorem.

If $\cal L_{\cal A,\bc}$ has distinct critical points $x_0$ and $x_0'$ in $\T{g}$, then these are local minima of $\cal L_{\cal A,\bc}$ since $\Hess\cal L_{\cal A,\bc}$ is positive definite at both $x_0$ and $x_0'$. Without loss of generality, assume $\cal L_{\cal A,\bc}(x_0')\le\cal L_{\cal A,\bc}(x_0)$. However, by strict convexity, this is impossible. Let $\gamma$ be the unique geodesic with $\gamma(0)=x_0$ and $\gamma(1)=x_0'$. Then $$\cal L_{\cal A,\bc}(\gamma(t))< (1-t)\cal L_{\cal A,\bc}(x_0)+t\cal L_{\cal A,\bc}(x_0')\le\cal L_{\cal A,\bc}(x_0)$$ for all $t\in(0,1]$, contradicting the fact that $x_0$ must be a local minimum. Hence $x_0=x_0'$ and $\cal L_{\cal A,\bc}$ has a unique critical point in $\T{g}$, denoted $x_0$. This proves property (2).

Since the components of $\cal H_g$ are totally geodesic in the WP metric, the same argument shows that the restriction $\cal L_{\cal A,\bc}|_H$ will have a unique critical point, denoted $x_H$, for each component $H$ of $\cal H_g$. This proves property (3) of the theorem. Thus, properties (1) through (3) of the theorem above are satisfied by the function $\cal L_{\cal A,\bc}$ for any value of $\bc$.

Let $S=\{H: H\text{ is a component of }\cal H_g\}\cup\{0\}$. For each pair $i,j\in S$ of distinct elements, there is an open dense subset $U_{i,j}$ of $\R_+^\cal A$ given by $$U_{i,j}=\left\{\bc\in\R_+^\cal A:\cal L_{\cal A,\bc}(x_i)\not=\cal L_{\cal A,\bc}(x_j)\right\}.$$ By the Baire Category Theorem, $\bigcap_{i\not=j}U_{i,j}$ is open and dense in $\R_+^\cal A$. Let $\bc'\in\bigcap_{i\not=j}U_{i,j}$. We now define $f:=\cal L_{\cal A,\bc'}$. Then, $f$ satisfies property (4).

Lastly, we wish to show that $f(S)$ is discrete. Choose a neighborhood $U_0$ of $x_0$ and $U_H$ of $x_H$, for each component $H$ of $\cal H_g$ which are mutually disjoint. Properness of $f$ then implies that $f(S)$ is discrete. This shows that $f$ satisfies property (5).
\end{proof}

\subsection{Relative Morse theory of the pair $(\T{g},\cal H_g)$}

The goal of this subsection is to prove Theorem \ref{M2}. The idea is that the Morse function $f$ found in Lemma \ref{Lem1} may be used to determine a handle decomposition of both $\cal H_g$ and $\T{g}\setminus\cal H_g$. For a reference on relative Morse theory, see e.g. \cite[Section 3]{Sharpe}.

\begin{namedtheorem}[Theorem \ref{M2}]
Let $g\ge 3$. The hyperelliptic complement $\T{g}\setminus\cal H_g$ has the homotopy type of a wedge $\displaystyle\bigvee_{i=1}^\infty S^n$ of infinitely many $n$-spheres, where $n=2g-5$.
\end{namedtheorem}

Note that since every curve of genus $g=2$ is hyperelliptic, $\T{2}\setminus\cal H_2=\varnothing$. The proof of Theorem \ref{M2} is similar to Mess's proof that the image of the period mapping on $\T{2}$ has the homotopy type of an infinite wedge of circles \cite[Proposition 4]{Mess}. We now prove Theorem \ref{M2}.

\begin{proof}
The idea behind relative Morse theory is that such a function as given by Lemma \ref{Lem1} can be used to determine a handle decomposition not only of $\cal H_g$, but of its complement $\T{g}\setminus\cal H_g$. Let $f$ be the function that satisfies the conclusion of Lemma \ref{Lem1}. We let $x_0$ denote the unique minimum point of $f$ in $\T{g}$. For each component $H$ of $\cal H_g$, let $x_H$ denote the unique critical point of $f|_H$. We refer to $x_0$ as a {\it critical point of $f$ of type I} and each $x_H$ are referred to as {\it critical points of $f$ of type II}. The values $c_0=f(x_0)$ and $c_H=f(x_H)$ are called {\it critical values of type I} and {\it II}, respectively. 

For $r$ a real number, let $X_r:=\{\cal X\in\T{g}: f(\cal X)\le r\}$. If $(c_0,c_0+\epsilon]$ contains no type II critical values, then $X_{c_0+\epsilon}\setminus\cal H_g$ is diffeomorphic to a 0-handle, i.e. a closed ball. Consider an arbitrary interval $[a,b]\subset\R$. If $[a,b]$ contains no critical value of type I or II of $f$, then $X_a\setminus\cal H_g$ is diffeomorphic to $X_b\setminus\cal H_g$. To see this, we can construct a vector field $V$ which is equal to $\grad(f)$ outside a neighborhood of $\cal H_g$ and such that $V|_{\cal H_g}$ is equal to $\grad(f|_{\cal H_g})$. The flow along this vector field gives the required diffeomorphism.

Let $x$ be a critical point of type II, and let $c=f(x)$. By Lemma \ref{Lem1}, the set of critical values of $f$ is discrete, so there exists some $\epsilon>0$ such that $[c-\epsilon,c+\epsilon]$ contains no other critical values of $f$. We wish to show that $X_{c+\epsilon}\setminus\cal H_g$ is diffeomorphic to $X_{c-\epsilon}\setminus\cal H_g$ with an $n$-handle attached, where $n=2g-5$ (see Figure 2).

Let $H$ be the component of $\cal H_g$ containing $x$. There exists a coordinate system $(u,y)\in \R^{2g-4}\times\R^{4g-2}$ in a neighborhood $U$ of $x$ such that \cite[3.3]{Sharpe}
\begin{enumerate}[topsep=0pt,itemsep=-1ex,partopsep=1ex,parsep=1ex]
\item $U\cap H$ is given by $u=0$,
\item $f=c+\|y\|^2$ on $U\cap H$.
\end{enumerate}
The coordinates $y$ are ``tangent'' coordinates to $H$ and the coordinates $u$ are ``normal'' coordinates to $H$. Note that since $H$ has complex dimension $2g-1$, it has real dimension $4g-2$.

\begin{figure}[h]
\includegraphics[scale=0.75]{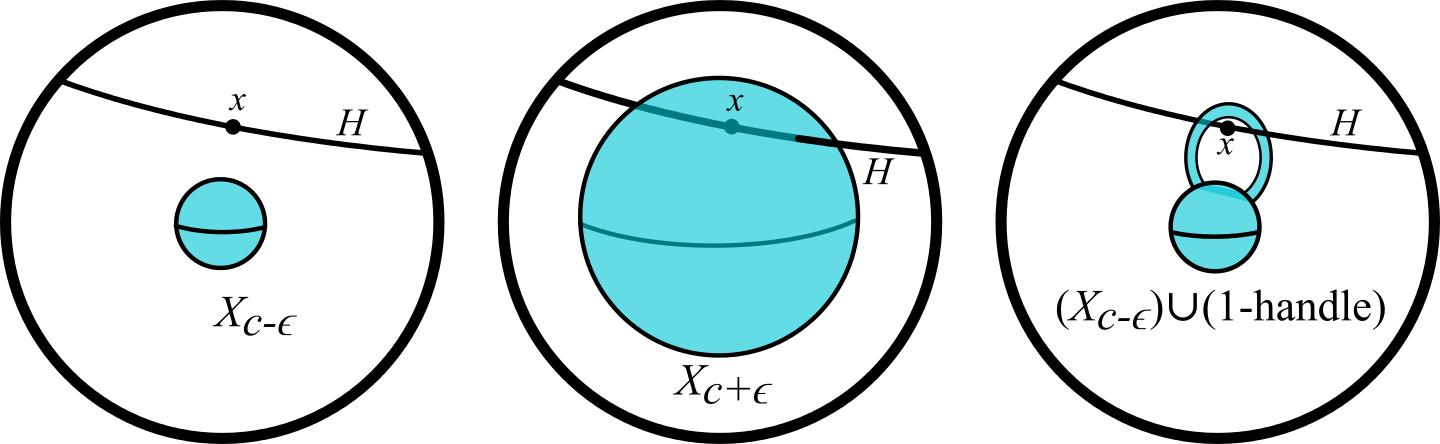}
	\caption{Start with $X_{c-\epsilon}$. As $c-\epsilon$ increases to $c+\epsilon$, the level set $X_{c+\epsilon}$ intersects exactly one more component $H$ of $\cal H_g$, the component containing the critical point $x$.}
\end{figure}

Then, $X_{c+\epsilon}\setminus\cal H_g$ is diffeomorphic to the union of $X_{c-\epsilon}\setminus\cal H_g$ and a tubular neighborhood of $$\{(u,0): \|u\|^2=\delta\},$$ for some small $\delta>0$. This tubular neighborhood deformation retracts to the $(2g-5)$-sphere $\{(u,0): \|u\|^2=\delta\}$. Hence, $\T{g}\setminus\cal H_g$ has a handle decomposition consisting of a 0-handle with infinitely many (one for each component of $\cal H_g$) $n$-handles attached, where $n=2g-5$.
\end{proof}

Let $\cal M_g^{nhyp}$ denote the moduli space of hyperelliptic curves of genus $g$. Since $\T{3}\setminus\cal H_3$ is a covering space for $\cal M_3^{nhyp}$, the moduli space $\cal M_3^{nhyp}$ has contractible universal cover and $\cal M_3^{nhyp}$ is aspherical. If $g\ge 4$ then $\pi_n(\cal M_g^{nhyp})$, where $n=2g-5>1$, is an infinite rank abelian group. In particular, $\cal M_g^{nhyp}$ is not aspherical for $g\ge 4$.

We can be even more precise. The components of the hyperelliptic locus $\cal H_g$ are enumerated by the set of cosets of the group $\HH{g}$ in $\Mod{g}$. Recall that $\HH{g}$ is the centralizer in $\Mod{g}$ of a fixed hyperelliptic involution $\tau\in\Mod{g}$. The group $\HH{g}$ is called the {\it hyperelliptic mapping class group of genus $g$}. If $\eta$ is another hyperelliptic involution, then the centralizers of $\tau$ and $\eta$ are conjugate in $\Mod{g}$.

\begin{cor}\label{HomotopyEquiv}
Let $g\ge 3$. There is a homotopy equivalence $$\T{g}\setminus\cal H_g\cong\bigvee_{[h]\in\Mod{g}/\HH{g}}\S^{2g-5}.$$ In particular, $$H_{2g-5}(\T{g}\setminus\cal H_g;\Z)\cong\Z[\Mod{g}/\HH{g}]$$ as $\Mod{g}$-modules.
\end{cor}

\begin{proof}
The mapping class group $\Mod{g}$ acts on the set of components of $\cal H_g$ by permutations. Then, there is a map \begin{align*}\Orb(H_0)&\to \Mod{g}/\Stab(H_0)\\ h\cdot H_0&\mapsto h\Stab(H_0)\end{align*} from the orbit $\Orb(H_0)$ of $H_0$ to the left coset space of the stabilizer $\Stab(H_0)$. It suffices to show that $\Stab(H_0)=\HH{g}$ and $\Mod{g}$ acts transitively on the set of components of $\cal H_g$.

First, since $H_0=\Fix(\tau)$, the mapping class $h\in\Stab(H_0)$ if and only if $$h\cdot\Fix(\tau)=\Fix(h\tau h^{-1})=\Fix(\tau).$$ Since no hyperelliptic curve can have two distinct hyperelliptic involutions, it must follow that $h\tau h^{-1}=\tau$ so $h\in\HH{g}$. Therefore, $\Stab(H_0)=\HH{g}$.

Secondly, if $H$ is any other component of $\cal H_g$, then $H=\Fix(\eta)$ for some hyperelliptic involution $\eta\in\Mod{g}$. Since hyperelliptic involutions in $\Mod{g}$ are conjugate, there exists some $h\in\HH{g}$ such that $$H=\Fix(\eta)=\Fix(h\tau h^{-1})=h\cdot\Fix(\tau)=h\cdot H_0.$$ Therefore, $\Mod{g}$ acts transitively on the set of components of $\cal H_g$.
\end{proof}

\section{The Parameter Space of Smooth Plane Curves}\label{section:PlaneCurves}

In this section, we prove Proposition \ref{principal}, showing that the cover of $\cal U_d$ determined by the subgroup $K_d$ of $\pi_1(\cal U_d)$ carries the structure of a principal fiber bundle. This will be critical in the proof of Theorem \ref{M1} in Section 4.2.

\subsection{Covers of $\cal U_d$ and principal fiber bundles}
The main result of this subsection is to prove Proposition \ref{principal}, exhibiting a cover of $\cal U_d$ as a principal fiber bundle over a certain subspace of $\T{g}$.

Associating each point of $\cal U_d$ to the curve it determines gives rise to a map $\varphi_d:\cal U_d\to\M_g$ into the moduli space of Riemann surfaces of genus $g(d)$, where $g=g(d):=\binom{d-1}{2}$ by the degree-genus formula. Let $\M_g^{pl}$ denote the image of this map. For $d\ge 4$, the locus $\M_g^{pl}\subsetneq\M_g$ and for $d=3$, $\M_1^{pl}=\M_1$.

There is a (disconnected) covering $\cal U_d^{mark}$ of $\cal U_d$ defined as follows. A point $(f,[h])\in\cal U_d^{mark}$ is an ordered pair consiting of $f\in\cal U_d$ and a homotopy class $[h]$ of orientation-preserving homeomorphisms $h:\Sigma_g\to V(f)$ of some fixed $\Sigma_g$ with the complex curve $V(f)$ given by $f(x,y,z)=0$.

Let $\pi_1(\cal U_d^{mark})$ be the fundamental group of a chosen component of $\cal U_d^{mark}$. Note that $\pi_1(\cal U_d^{mark})\cong K_d$.

\begin{rem}
There is a commutative diagram
\begin{diagram}
\cal U_d^{mark}		&\rTo^{\tilde{\varphi}_d}		&\T{g}\\
\dTo				&						&\dTo_\pi\\
\cal U_d			&\rTo_{\varphi_d}				&\M_g
\end{diagram}
The map $\varphi_d:\cal U_d\to\cal M_g$ lifts to a map $\tilde{\varphi}_d:\cal U_d^{mark}\to\T{g}$ into Teich\"muller space defined by $$\varphi_d: (f,[h])\mapsto [V(f),h].$$ Let $\T{g}^{pl}$ denote the image of $\varphi_d$.
\end{rem}

Recall that a {\it principal $G$-bundle} is a fiber bundle $P\to X$ with a $G$-action that acts freely and transitively on the fibers.

\begin{prop}\label{principal}
For $d\ge 4$, the map $\tilde\varphi_d:\cal U_d^{mark}\to\T{g}^{pl}$ is a principal $\PGL_3(\C)$-bundle.
\end{prop}

\begin{proof}

First, $\PGL_3(\C)$ acts on $\cal U_d^{mark}$ by $g\cdot (f,[h])=(g\cdot f, [g\circ h])$ where $g\cdot f$ denotes the action of $g$ on polynomials $f(x,y,z)$, by acting on the triple of variables $(x,y,z)$. This induces a map $g:V(f)\to V(g\cdot f)$ and $g\circ h$ is the composition of this map with the marking $h:\Sigma_g\to V(f)$.

This action is free. Indeed, if $g\cdot (f,[h])=(f,[h])$ then $g\cdot f=f$ and $[g\circ h]=[h]$. Thus $g$ induces an automorphism on the curve $V(f)$. Moreover, this automorphism acts trivially on the marking, hence trivially on $H_1(V(f);\Z)$. An automorphism of $V(f)$ acting trivially on homology must be the identity \cite[Theorem 6.8]{FM}. The fixed set of any automorphism of $\P^2$ is a linear subspace, so any $g\in\PGL_3(\C)$ point-wise fixing a smooth quartic curve must be the identity automorphism.

Next, we show that this action is transitive on fibers. It suffices to show that if $\tilde\varphi_d(f_1,[h_1])=\tilde\varphi_d(f_2,[h_2])$, then the $(f_i,[h_i])$ lie in the same $\PGL_3(\C)$-orbit. By assumption, $[V(f_1),h_1]=[V(f_2),h_2]$ and there is some biholomorphism $\psi:V(f_1)\to V(f_2)$ such that $[\psi\circ h_1]=[h_2]$. Then the pullback of the hyperplane bundle $H$ along the embeddings $e_i:V(f_i)\to\P^2$ gives line bundles $L_i:=e_i^*(H)$ on $V(f_i)$ of degree $d$ with $h^0(L_i)=3$.

A $g^r_d$ {\it line bundle} is a line bundle $L\to C$ such that $\deg(L)=d$ and $h^0(L)\ge r+1$. Smooth plane curves have a unique $g^2_d$ given by the pullback of the hyperplane bundle \cite[Theorem 3.13]{Ser}. Therefore, $L_1$ and $\psi^*L_2$ are isomorphic line bundles on $V(f_1)$.

For any smooth curve $C$, there is a correspondence between maps $C\to\P^r$ up to the action of $\PGL_{r+1}(\C)$ and pairs $(L,V)$ where $L$ is a line bundle over $C$ and $V\subset H^0(C;L)$ is an $(r+1)$-dimensional subspace. The fact that there is a unique line bundle $L$ on $V(f_1)$ with $h^0(L)\ge 3$ implies that there is only one such map $V(f_1)\to\P^2$ up to the action of $\PGL_3(\C)$. Therefore, the two embeddings $e_1$ and $e_2\circ\psi$ are equivalent up to the action of $\PGL_2(\C)$, i.e. there is some $g\in\PGL_2(\C)$ such that $g\circ e_1=e_2\circ\psi$. This implies that $g\cdot f_1=f_2$ and $g:V(f_1)\to V(f_2)$ coincides with $\psi$. Thus, $(f_1,[h_1])$ and $(f_2,[h_2])$ lie in the same $\PGL_3(\C)$-orbit.

Finally, it remains to prove local triviality. This is a consequence of a much more general fact that if $G$ acts on a manifold $P$ freely such that $P/G$ is a manifold, then $q:P\to P/G$ is locally trivial. Indeed, a local trivialization of $q:P\to P/G$ can be built over any contractible subset $U$ by first taking a section $\sigma:U\to P$ and defining $\varphi:q^{-1}(U)\to U\times G$ by $\varphi(x)=(q(x),g(x))$, where $g(x)\in G$ is the unique element such that $x=g(x)\cdot\sigma(q(x))$.
\end{proof}

\begin{prop}
Let $d\ge 3$ and $g=\binom{d-2}{2}$. The space $\cal U^{mark}_d$ has finitely many components. Consequently, $\T{g}^{pl}$ has finitely many components.
\end{prop}

\begin{proof}
A single component of $\cal U_d^{mark}$ is the connected covering space of $\cal U_d$ corresponding to $K_d$. Hence, its deck transformation group is the image of the homomorphism $\rho_d:\pi_1(\cal U_d)\to\Mod{g}$. The components of $\cal U_d^{mark}$ are enumerated by the cosets of $\im(\rho_d)$ in $\Mod{g}$. It was shown in and \cite[Theorem A]{Salt2} that the index $[\Mod{g}:\im(\rho_d)]<\infty$.
\end{proof}

\subsection{The kernel of the geometric monodromy of the universal quartic} In this subsection, we prove Theorem \ref{M1}.

\begin{namedtheorem}[Theorem \ref{M1}]
The group $K_4$ is isomorphic to $F_\infty\times\Z/3\Z$, where $F_\infty$ is an infinite rank free group. Moreover, $F_\infty$ has a free generating set in bijection with the set of cosets of the hyperelliptic mapping class group $\HH{3}$, and $$H_1(K_4;\Q)\cong \Q[\Mod{3}/\HH{3}]$$ as $\Mod{3}$-modules.
\end{namedtheorem}

\begin{proof}[Proof of Theorem \ref{M1}] Classically, $\T{3}^{pl}$ is exactly the complement of the hyperelliptic locus $\cal H_3$ in $\T{3}$: the canonical map $C\to\P^2$ is an embedding precisely when $C$ is nonhyperelliptic \cite[Pages 246-7]{GH}. Consider the following principal fiber bundle.
\begin{diagram}
\PGL_3(\C)	&\rTo	&	\cal U_4^{mark}\\
			&		&	\dTo_{\varphi_4}\\
			&		&	\T{3}\setminus\cal H_3
\end{diagram}
Because $\rho_4:\pi_1(\cal U_4)\to\Mod{3}$ is surjective \cite[Proposition 6.3]{K}, $\cal U_4^{mark}$ is connected.

By Theorem \ref{M2}, $\T{3}\setminus\cal H_3$ is homotopy equivalent to an infinite wedge of circles and, since $\PGL_3(\C)$ is connected, there must exist some continuous section $\sigma:\T{3}\setminus\cal H_3\to\cal U_4^{mark}$. Because $\varphi_4$ is a principal $\PGL_3(\C)$-bundle, the existence of such a section implies that $\cal U_4^{mark}$ is homeomorphic to $\PGL_3(\C)\times(\T{3}\setminus\cal H_3)$, and so
\begin{align}\label{kernelcalculation}
\pi_i(\cal U_4^{mark})=\begin{cases}\Z/3\Z\times F_\infty, &\text{for }i=1\\ \pi_i(\PGL_3(\C)), &\text{for }i>1.\end{cases}
\end{align} This also shows that $\pi_i(\cal U_4)\cong\pi_i(\PGL_3(\C))$ for $i\ge 2$.

We now wish to show that $H_1(K_4;\Q)$ is isomorphic to $\Q\left[\Mod{3}/\HH{3}\right]$ as $\Mod{3}$-modules. The calculation of $K_4\cong\pi_1(\cal U_4^{mark})$ in equation \ref{kernelcalculation} shows that the projection $$\cal U_4^{mark}\xrightarrow{\cong}\PGL_3(\C)\times(\T{3}\setminus\cal H_3)\to\T{3}\setminus\cal H_3$$ induces an isomorphism $$H_1(K_4;\Q)\cong H_1(\T{3}\setminus\cal H_3;\Q).$$ The action of $\Mod{3}$ on $\cal U_4^{mark}$ commutes with the projection map $$\cal U_4^{mark}\to\T{3}\setminus\cal H_3,$$ so that the above isomorphism of $\Q$-vector spaces is an isomorphism of $\Mod{3}$-modules.

The group $H_1(\T{3}\setminus\cal H_3;\Z)$ is the free abelian group on the set of cycles in $\T{3}\setminus\cal H_3$ represented by meridians around the components of the hyperelliptic locus $\cal H_3$; that is, the boundaries of disks transversely intersecting $\cal H_3$ in a single point. Such cycles are in bijection with the cosets of $\Mod{3}/\HH{3}$ (see proof of Corollary \ref{HomotopyEquiv}). This bijection commutes with the action of $\Mod{3}$ and therefore this $\Mod{3}$-module is isomorphic to the permutation representation $\Q[\Mod{3}/\HH{3}]$.
\end{proof}

The following table shows $\pi_i(\cal U_4)\cong\pi_i(\PGL_3(\C))$ for small values of $i\ge 2$ (c.f. \cite[Introduction]{Mam}, where we have used the fact that $\SL_3(\C)$ covers $\PGL_3(\C)$ and is homotopy equivalent to ${\rm SU}(3)$).

\begin{center}
\begin{tabular}{|c|c|c|c|c|c|c|c|c|c|c|c|}
\hline
$i$						&2	&3		&4	&5		&6			&7	&8			&9			&10			&11			&12\\
\hline
$\pi_i(\cal U_4)$		&0	&$\Z$	&0	&$\Z$	&$\Z/6\Z$	&0	&$\Z/12\Z$	&$\Z/3\Z$	&$\Z/30\Z$	&$\Z/4\Z$	&$\Z/60\Z$\\
\hline
\end{tabular}
\end{center}

\bibliography{bibliography}
\bibliographystyle{alpha}

\end{document}